\documentclass[12pt,reqno]{amsart}
\usepackage[foot]{amsaddr}
\usepackage[margin=2cm]{geometry}
\usepackage[utf8]{inputenc}
\usepackage{enumerate,enumitem,amsmath,amsthm,latexsym,amssymb,statmath,comment,times,bm}
\usepackage{mathtools}
\usepackage{dsfont}
\usepackage{mathrsfs}

\usepackage{aligned-overset}

\usepackage{hyperref}
\usepackage[capitalise]{cleveref}
\crefname{algocfline}{Algorithm}{Algorithms}
\crefname{equation}{}{}
\crefname{prop}{Proposition}{Propositions}
\crefname{enumi}{}{}
\crefname{figure}{Figure}{Figures}
\crefname{claim}{Claim}{Claims}
\crefname{subsection}{Subsection}{Subsections}

\usepackage{mathtools}
\newtheorem{theorem}{Theorem}[section]

\newtheorem{lemma}[theorem]{Lemma}
\newtheorem{claim}[theorem]{Claim}

\theoremstyle{definition}

\newcommand{\prob}[1]{\mathbb{P}\left[#1\right]}
\newcommand{\variance}[1]{\mathbb{V}\left[#1\right]}
\newcommand{\expec}[1]{\mathbb{E}\left[#1\right]}

\newcommand{\condvariance}[2]{\mathbb{V}\left[#1 \;\middle|\; #2\right]}

\def\cE{\mathcal{E}}

\def\tY{\widetilde{Y}}
\def\tZ{\widetilde{Z}}

\def\hX{\widehat{X}}
\def\hZ{\widehat{Z}}
\def\hY{\widehat{Y}}

\def\hc{\hat{c}}

\def\eps{\varepsilon}

\usepackage{xcolor}

\usepackage[normalem]{ulem}

\usepackage{calc}
\usepackage{tikz}
\usetikzlibrary{arrows}
\usetikzlibrary{decorations}
\usetikzlibrary{shapes.misc}
\usetikzlibrary{decorations.pathreplacing}
\usetikzlibrary{calc}
\usetikzlibrary{math}

\usepackage{enumerate}
\usepackage{enumitem}

\newenvironment{proofclaim}[1][Proof of Claim]{\begin{proof}[#1]}{\end{proof}}

\newcommand{\ind}[1]{\mathds{1}_{\{#1\}}}
\newcommand{\indev}[1]{\mathds{1}_{#1}}

\title[A central limit theorem for the order of the giant component and $k$-core]{A short proof of a central limit theorem for the order of the giant component and $k$-core}
\author{Michael Anastos$^1$} 
\address{$^1$Institute of Science and Technology Austria (ISTA), 3400 Klosterneurburg, Austria}
\author{Joshua Erde$^2$}
\address{$^2$School of Mathematics, University of Birmingham, Birmingham, UK}
\author{Mihyun Kang$^3$}
\address{$^3$Institute of Discrete Mathematics, Graz University of Technology, Austria}
\author{Vincent Pfenninger$^4$}
\address{$^4$Department of Mathematics, London School of Economics, London, United Kingdom}

\email{michael.anastos@ist.ac.at}

\email{j.erde@bham.ac.uk}

\email{kang@math.tugraz.at}

\email{v.pfenninger@lse.ac.uk}

\date{\today}


\begin{document}
\begin{abstract}
In this note we outline a new and simple approach to proving central limit theorems for various `global' graph parameters which have robust `local' approximations, using the Efron--Stein inequality, which relies on a combinatorial analysis of the stability of these approximations under resampling an edge. As an application, we give short proofs of a central limit theorem for the order of the giant component and of the $k$-core for sparse random graphs.
\end{abstract}

\maketitle

\section{Introduction}
In discrete probability, and in particular the theory of random graphs, there has been much interest in establishing \emph{central limit theorems} for various parameters of interest. More precisely, given a sequence $(Z_n)_{n \in \mathbb{N}}$ of random variables defined on the probability spaces $G(n,p)$, we are interested in when 
\[
X_n \coloneqq \frac{Z_n - \expec{Z_n}}{\sqrt{\variance{Z_n}}} \quad\indist\quad \mathcal{N}(0,1),
\]
as $n \rightarrow \infty$, that is, when $X_n$ converges in distribution to the standard normal distribution. Here, by $G(n,p)$ we denote the binomial random graph, that is, the random graph on $n$ vertices where each edge appears independently with probability $p$.

Early work on this topic established such results for various types of \emph{subgraph counts} in random graphs, and many different techniques have been developed to deal with these parameters --- the method of moments \cite{K84,KR83,R88}, Stein's method \cite{Barbour1989}, $U$-statistics and related methods \cite{J94,NW88}, and martingale methods \cite{BJKR90,dJ96}. Moreover, these methods can also be used to show \emph{multivariate central limit theorems} for finite collections of subgraph counts, establishing \emph{joint normality} of the rescaled random variables. In particular, this implies that, for appropriate ranges of $p$ and excluding some pathological cases, any finite linear combination of these parameters also satisfies a central limit theorem, which in turn implies central limit theorems for a broad class of `local' graph parameters --- those which can be determined by looking at bounded-order subgraphs of $G(n,p)$.

More recently, central limit theorems have been proven for  more `global' graph parameters. An early example of this comes from the work of Pittel and Wormald on the order of the giant component in a supercritical sparse random graph \cite{PW05,PW08} which used enumerative/analytic methods to establish a central limit theorem. More modern proofs \cite{BCOK14,BR12} involve the analysis of a graph exploration process using martingale methods. Central limit theorems have also been shown for the order of the $k$-core \cite{COCKS19,JL08} and very recently the rank and the matching number \cite{GKSS23,GKSS24,K17}, both of which analyse certain peeling processes on $G(n,p)$, using martingale methods or the differential equations method, respectively.

Recent work of the authors \cite{AEKP25} establishes a central limit theorem for another `global' graph parameter, the \emph{circumference}, which is the  length of the longest cycle. Rather than analysing the evolution of some random process, the proof instead employs a \emph{perturbative} argument using more elementary, combinatorial methods based on the Efron--Stein inequality, using the joint normality of certain subgraph counts as a `black-box'. However, the application in \cite{AEKP25} is still quite complicated and technical, due in part to certain `non-monotonicity' issues that arise within the structural analysis in the particular case of the circumference. 

The purpose of this note is to demonstrate the simplicity of the method developed by the authors in~\cite{AEKP25} by giving some short proofs of known results of central limit theorems ---  for the order of the giant component and the $k$-core in sparse random graphs in their respective supercritical regimes. 
\begin{theorem}\label{t:giant}
Let $c>1$, let $p=\frac{c}{n}$ and let $L$ be the largest component in $G(n,p)$. Then $|L|$ satisfies a central limit theorem.
\end{theorem}

\begin{theorem}\label{t:core}
Let $k \geq 3$, and let $\hc_k$ be the unique real number in $(k, \infty)$ such that 
\begin{equation}\label{e:hc_k}
\mathbb{P}\left[\mathrm{Poisson}\left(\hc_k\right) \leq k-1 \right] = (e\hc_k)^{-1}.
\end{equation}
Let $c> \hc_k$, let $p=\frac{c}{n}$ and let $K$ be the $k$-core of $G(n,p)$. Then $|V(K)|$ satisfies a central limit theorem.
\end{theorem}  

It is straightforward to show that $\hc_k$ is well defined and that $\hc_k = k + \sqrt{2 k \log k} + O(\sqrt{k})$ as $k \rightarrow \infty$.
For comparison, the $k$-core threshold $c_k$, i.e., the threshold value  of $c$ for the existence of a (non-empty) $k$-core, is shown in \cite{PSW96} to be asymptotically $k + \sqrt{k\log k} + O(\log k)$, and so, while $\hc_k$ in \Cref{t:core} is not quite sharp, it is very close to $c_k$. We also note that with a bit more care, it is possible to extract quantitative bounds on the variances in \Cref{t:core,t:giant}, as in \cite{AEKP25}.

The proofs for both \cref{t:core,t:giant} follow the same rough idea: We take a sequence of `local' approximations to the order of the giant component (or $k$-core) by taking, for each vertex $v \in V(G)$, an approximation of the indicator of the event that $v$ lies inside the giant component (or inside the $k$-core, respectively), based on the `local' structure of $G \sim G(n,p)$ around $v$, taking into account larger and larger neighbourhoods as our sequence develops.

A relatively standard application of Stein's method will show that each of these local approximations satisfies a central limit theorem and so, modulo some small technical details, it is sufficient to bound in a sufficiently strong manner the errors of these approximations, which we do using the Efron--Stein inequality to bound the variance of these errors.

Roughly, the Efron--Stein inequality allows us to bound the variance of a graph parameter by bounding the expected effect of \emph{resampling a single edge}. A simple combinatorial analysis will show that, in both cases, the effect of resampling an edge $e=\{u,v\}$ on the error of the approximation is in some way `localised' within the components $C_u$ and $C_v$  containing $u$ and $v$ of the \emph{remainder} (or \emph{mantle}), the graph $R$ remaining after we delete the giant component (or $k$-core, respectively). However, simple first moment calculations demonstrate that, for sufficiently large $c$, the component structure of $R$ resembles that of a \emph{subcritical} random graph, and so the expected effect of resampling an edge is `small'.

In the case of the giant component, this holds for any $c>1$, whilst for the $k$-core we require that $c$ is sufficiently large (i.e., larger than $\hc_k$). In fact, heuristically this should hold already for any $c$ above the $k$-core threshold $c_k$, although there does not seem to be a simple proof of this fact, and it is unclear if it follows directly from known results in the literature. However, we note that such a result, stating that the mantle has a subcritical distribution for a certain range of $p$, is all that would be required to extend the proof of Theorem \ref{t:core} to this range of $p$.
We note also that, whilst the statement and proof of Theorem \ref{t:core} extends without change to the case $k=2$, in fact similar calculations as that for the giant component show that the threshold for the appearance of a linear $2$-core coincides with that for the giant component, and that the `mantle' of the $2$-core will have a subcritical distribution for any $c>1$. It is then possible to show an analogue of Theorem \ref{t:core} for $k=2$ holds for any $c>1$ using similar methods, although we omit the details.

The case of the circumference, dealt with in \cite{AEKP25}, follows broadly the same strategy, but has to deal with significant technical difficulties which do not arise in the case of the giant component or $k$-core. These difficulties are mostly related to the method of `approximating' the vertex set of a longest cycle $C$, in itself not an obviously straightforward task. In particular, the fact that this approximation is not monotone under edge deletions, which serendipitously is the case for the applications in this paper, leads to sophisticated combinatorial analysis to determine the effect of resampling an edge. Furthermore, whilst it is relatively straightforward to estimate the variance of $|L|$ and $|V(K)|$, which is necessary to the application of Stein's method, the corresponding calculation for $|C|$ is far from trivial. Indeed, the resolutions of both of these problems take almost as long, if not longer, than the current paper.

Looking forward, we note that, in many cases, it can be shown that multiple parameters of interest simultaneously satisfy a \emph{multidimensional central limit theorem}. For example, Pittel and Wormald \cite{PW05,PW08} showed that the order of the giant component and the order of the $2$-core, when appropriately rescaled, are jointly Gaussian. It would be interesting to know if our methods could be extended to the multidimensional case, reducing then the derivation of a multidimensional central limit theorem to the combinatorial analysis of the simultaneous effect of resampling an edge on some collection of parameters.

\section{Preliminaries}
In this paper, we consider a \emph{component} to be a maximal \emph{set of vertices} that induces a connected subgraph. For a graph $G$ and a vertex set $A \subseteq V(G)$, we define $G[A]$ to be the \emph{induced subgraph of $G$ with respect to $A$}, that is, the graph $G[A]$ with vertex set $V(G[A]) = A$ and edge set $E(G[A]) = \{\{x,y\} \in E(G) \colon x, y \in A\}$. Moreover, we let $G - A \coloneqq G[V(G) \setminus A]$.
A \emph{rooted graph} is a pair $(H,r)$ where $H$ is a graph and $r \in V(H)$. The \emph{radius} of a rooted graph $(H,r)$ is the maximum distance of a vertex in $H$ to the root $r$ (in particular, if $H$ is disconnected, the radius of $(H,r)$ is $\infty$). Two rooted graphs $(H_1,r_1)$ and $(H_2,r_2)$ are \emph{isomorphic} (written $(H_1,r_1) \cong (H_2,r_2)$) if there exists a bijection $f \colon V(H_1) \rightarrow V(H_2)$ such that $f(r_1) = r_2$ and $f(u)f(v) \in E(H_2)$ if and only if $uv \in E(H_1)$ for all $u,v \in V(H_1)$.

The \emph{$k$-core} of a graph $G$ is defined to be the maximal subgraph $K\subseteq G$ of minimum degree at least $k$, which is easily seen to be unique. A useful alternative characterisation of the $k$-core is given by the following \emph{peeling process}. Starting with $G$, we recursively delete some vertex of degree less than $k$. Again it is easy to see that the resulting graph $K$ is the $k$-core of $G$. We note that, for any vertex $w$ which lies in a component $C_w$ of $G - V(K)$, there is some \emph{$k$-degenerate ordering} of $C_w$, that is, an ordering of the vertices of $C_w$ such that each vertex is adjacent to fewer than $k$ vertices which appear before it or in~$K$.

\subsection{Probabilistic inequalities}

We use the following lemma, proved in \cite[Lemma 3.6]{AEKP25}, which bounds the variance of `local' graph parameters using the Efron--Stein inequality \cite{ES81} by considering the effect of resampling an edge.

\begin{lemma}\label{lem:ESapp}
Let $G \sim G(n,p)$ be a random graph, let $\{\psi_G(v) \colon v \in V(G)\}$ be a set of (isomorphism-invariant) random variables and let $Z \coloneqq \sum_{v \in V(G)} \psi_G(v)$. Let $f \in \binom{V(G)}{2}$ be arbitrary and let $G^+ = G+f$ and $G^-=G-f$. Suppose that $M>0$ is such that $|\psi_G(v)| \leq M$ for each $v \in V(G)$ and let 
\[
    D \coloneqq \{ v \in V(G) \colon \psi_{G^+}(v) \neq \psi_{G^-}(v)\}.
\]
Then
\[
    \variance{Z} \leq 2M^2 p(1-p) n^2 \expec{|D|^2}.
\]
\end{lemma}

We will also use a probabilistic form of Minkowski's inequality, see for example \cite[Equation (19.4)]{B95}.
\begin{theorem}[Minkowski's inequality for variances] \label{thm:prob_Minkowski}
    For jointly distributed random variables $Y$ and $Z$ with $\expec{Y^2},\expec{Z^2},|\expec{Y\cdot Z}| < \infty$,
    \[
    \variance{Y+Z} \leq \left(\sqrt{\variance{Y}} + \sqrt{\variance{Z}}\right)^2.
    \]
\end{theorem}
Furthermore, under the same conditions, \cref{thm:prob_Minkowski} also implies\footnote{To prove \cref{e:prob_CS}, we may assume without loss of generality that $\variance{Z} \leq \variance{Y}$, then write $Y = Z + (Y - Z)$ and apply \cref{thm:prob_Minkowski}.}
    \begin{align}\label{e:prob_CS}
    \left| \variance{Y} - \variance{Z}\right| \leq  \variance{Y-Z} + 2\left(\variance{Y-Z}\right)^{\frac{1}{2}} \min \{ \variance{Y}^{\frac{1}{2}},\variance{Z}^{\frac{1}{2}}\}.
    \end{align}
When applying \cref{e:prob_CS}, we further upper bound the minimum by one of its terms, as convenient.
We note that equivalent statements can also be derived from a probabilistic version of the Cauchy-Schwarz inequality, as in \cite[Theorem 3.1]{AEKP25}.

In the proof we will consider sequences of approximations which are themselves defined on the sequence $(G(n,p))_{n\in \mathbb{N}}$ of random graphs. For this reason, when talking about limits, it will be useful to have a version of Slutsky's Theorem \cite{S26} for `double limits'. It can easily be checked that the proof carries over into this setting.

Given a doubly-indexed sequence $(a_{\ell,n})_{\ell,n \in \mathbb{N}}$ of real numbers and $a \in \mathbb{R}$ we write
\begin{equation}\label{e:doublelimit}
   \lim_{\substack{ \ell,n \to \infty \\  n \gg \ell}} a_{\ell,n} = a,
\end{equation}
if there exists a function $N \colon \mathbb{N} \rightarrow \mathbb{N}$ such that for every $\eps > 0$ there exists $\ell_0$ such that for all $(\ell,n) \in \mathbb{N}^2$ with $\ell \geq \ell_0$ and $n \geq N(\ell)$ we have $|a_{\ell,n} - a| < \eps$. Given a doubly-indexed sequence $(X_{\ell,n})_{\ell,n \in \mathbb{N}}$ of real random variables and a real random variable $X$, we write $X_{\ell,n} \indist X$ as $\ell,n \to \infty$ if, for every $x \in \mathbb{R}$ at which $x \mapsto \prob{X \leq x}$ is continuous, we have
\[
   \lim_{\substack{ \ell,n \to \infty \\  n \gg \ell}}  \prob{X_{\ell,n} \leq x} = \prob{X \leq x}.
\]
Similarly, we write $X_{\ell,n} \inprob X$  as $\ell,n \to \infty$ if for all $\delta >0$,
\[
   \lim_{\substack{ \ell,n \to \infty \\  n \gg \ell}}  \prob{|X_{\ell,n} - X| \geq \delta} = 0.
\]
\begin{theorem}[Slutsky's Theorem \cite{S26}] \label{thm:Slutsky}
    Let $(X_{\ell,n})_{\ell,n \in \mathbb{N}}, (Y_{\ell,n})_{\ell, n \in \mathbb{N}},$ and $(Z_{\ell,n})_{\ell,n \in \mathbb{N}} $ be doubly-indexed sequences of real random variables. Let $Y$ be a real random variable and let $x,z \in \mathbb{R}$. If $X_{\ell,n} \inprob x$, $Y_{\ell,n} \indist Y$ and $Z_{\ell,n} \inprob z$ as $\ell,n \to \infty$, then $$X_{\ell,n} + Y_{\ell,n}\cdot Z_{\ell,n}\quad\indist\quad x + zY, \qquad \text{ as $\ell,n \to \infty$}.$$
\end{theorem}

\subsection{A central limit theorem for weighted neighbourhood sums}\label{subsec:weighted neighbourhoods}
We will also need the following general result, proved in \cite[Lemmas 5.2 and 7.6]{AEKP25}, which asserts a central limit theorem for a range of graph parameters which have a representation as a weighted sum of neighbourhood counts.

More precisely, given $(\ell,n) \in \mathbb{N}^2$ let $\mathcal{H}$ be the set of all rooted graphs $(H,r)$ with radius
at most $\ell$ and at most $n$ vertices and let $\mathcal{T}$ be the set of rooted trees $(T,r) \in \mathcal{H}$ with $1 \leq v(T) \leq t(c,\ell)$, for some appropriately large function $t$ of $c$ and $\ell$. A \emph{weighted $\ell$-neighbourhood count} is a random variable of the form
\[
X = \sum_{v \in V(G)} \sum_{(H,r) \in \mathcal{H}} \alpha_{(H,r)} \mathds{1}_{\{(G[B_G(v,\ell)],v) \cong (H,r)\}}
\]
 where $\alpha_{(H,r)} \in [0,1]$ for each $(H,r) \in \mathcal{H}$ and 
 $B_G(v, \ell)$ is the set of vertices in $G$ at distance at most $\ell$ from $v$.
Moreover, we define the \emph{truncation} of $X$ to be
\[
Y=\sum_{v \in V(G)}\sum_{(T,r) \in \mathcal{T}} \alpha_{(T,r)} \mathds{1}_{\{(G[B_G(v,\ell)],v) \cong (T,r)\}}.
\]

\begin{lemma}\label{l:trunc}
Let $c>1$ and let $p=\frac{c}{n}$. If $X$ is a weighted $\ell$-neighbourhood count and $Y$ is the truncation of $X$, then the following hold.
\begin{enumerate}[label = \upshape{(\roman*)}]
\item\label{i:var} $\variance{X-Y} \leq f_1(n,\ell) \cdot n$, where $\displaystyle \lim_{\substack{ \ell,n \to \infty \\  n \gg \ell}} f_1(n,\ell) = 0$.
\item\label{i:CLT} If $\variance{Y} = \omega\left(n^{\frac{2}{3}}\right)$, then $Y$ satisfies a central limit theorem.
\end{enumerate}    
\end{lemma}

For a proof see \cite[Lemmas 5.2 and 7.6]{AEKP25}.  We note that this follows the classic approach to convergence via Stein's method, which originated in \cite{Barbour1989}, where it was applied to the problem of counting induced subgraphs and more generally to random variables which can be decomposed as a sum of \emph{semi-induced properties}, such as weighted neighbourhood counts, under some mild technical restrictions (see \cite[Theorem 5]{Barbour1989}). This version of Stein's method is by now a standard tool in the area, and the statement of \cref{l:trunc}, in particular part \ref{i:CLT}, should come as no surprise to experts in the field.

\subsection{Typical properties of the remainder and mantle}

We will require the following properties of a supercritical sparse random graph which follow from standard arguments, but for completeness we include a short proof. In this paper, $\log$ refers to the natural logarithm.

\begin{lemma}\label{l:basicprop}
Let $c >1$ and let $p=\frac{c}{n}$. Let $G \sim G(n,p)$, $e \coloneqq \{u,v\} \in \binom{V(G)}{2}$, $G^- \coloneqq G - e$, and $G^+ \coloneqq G + e$. 
Let $\cE$ be the event that for $* \in \{-,+\}$, we have that $G^*$ has a unique component $L^*$ of linear order and all other components have order at most $\log^4 n$. Then $\prob{\cE} = 1 - o\left(\frac{1}{n^2}\right)$.

Furthermore, let $R^- = G^- - L^-$ and for each $v \in V(R^-)$, we let $C_v^-$ be the component of $R^-$ containing $v$ and $C_v^- \coloneqq \varnothing$ for $v \in L^-$. Then for $W \coloneqq C_u^- \cup C_v^-$, we have
\[
\expec{|W|^2 \ind{|W| \geq \ell} \indev{\cE}} \leq f_2(n,\ell),
\]
where $\displaystyle \lim_{\substack{ \ell,n \to \infty \\  n \gg \ell}} f_2(n,\ell) = 0$.
\end{lemma}
\begin{proof}
The first part of the lemma is a straightforward modification of a standard argument that appears in any textbook on random graphs (see, e.g., \cite{FK16}).

Furthermore, by the first part we can assume in the second part that all components of $R^-$ have order at most $\log^4 n$ and then calculate
\begin{align*}
\expec{|W|^2 \ind{|W| \geq \ell} \indev{\cE}} &\leq \sum_{j = \ell}^{\log^4 n} j^2 \binom{n-2}{j-2} j^{j-2} p^{j-2}(1-p)^{j(n-j)} + o(1) \\
&\leq \sum_{j = \ell}^{\log^4 n} j^2 e^2 \left( c e^{1-c + o(1)} \right)^{j-2} +o(1) := f_2(n,\ell)
\end{align*}
and the conclusion follows since $ce^{1-c} < 1$ for $c \neq 1$.
\end{proof}

We will also need similar statements about the properties of the mantle.
\begin{lemma}\label{lem:kcoreprop}
Let $k \geq 3$, let $c> \hc_k$, where $\hc_k$ is defined as in \eqref{e:hc_k}, and let $p = \frac{c}{n}$. Let $G \sim G(n,p)$, let $e \coloneqq \{u,v\} \in \binom{V(G)}{2}$, and let $G^- \coloneqq G - e$ and $G^+ \coloneqq G + e$. For each $* \in \{-,+\}$, let $K^*$ be the $k$-core of $G^*$ and let $R^* \coloneqq G^* - V(K^*)$.
Let $\cE$ be the event that for $* \in \{-,+\}$, all components of $R^*$ have order at most $\log^4 n$. Then the following hold.
\begin{enumerate}[label=\upshape{(\roman*)}]
    \item $\prob{\cE} = 1 - o\left(\frac{1}{n^2}\right)$. \label{i:kcp1}
    \item We have $\hc_k \geq c_k$ where $c_k$ is the $k$-core threshold. \label{i:kcp2}
    \item Furthermore, for each $w \in V(R^-)$, we let $C_w^-$ be the component of $R^-$ containing $w$ and $C_w^- \coloneqq \varnothing$ for $w \in K^-$. Then for $W \coloneqq C_u^- \cup C_v^-$, we have
    \[
        \expec{|W|^2 \ind{|W| \geq \ell} \indev{\cE}} \leq f_3(n,\ell),
    \]
    where $\displaystyle \lim_{\substack{ \ell,n \to \infty \\  n \gg \ell}} f_3(n,\ell) = 0$. \label{i:kcp3}
\end{enumerate}

\end{lemma}
\begin{proof}
We first note that the function $f(x)= x \cdot \mathbb{P}(\mathrm{Poisson}(x)\leq k-1)$ is strictly decreasing for $x > k$. Indeed, for $y >x > k$,  as $f(y)=\sum_{i=1}^{k}e^{-y}y^i/(i-1)! \leq e^{-(y-x)}(y/x)^k f(x)$ we have that
\begin{align*}
    \frac{f(y)}{f(x)}\leq  e^{-(y-x)}(y/x)^k =e^{-(y-x)}(1+(y-x)/x)^k \leq  e^{-(y-x)+k(y-x)/x}=e^{-(y-x)(1-k/x)}<1.
\end{align*}
By Poisson convergence (see, e.g.,~\cite[Theorem 3.6.1]{D19}) and the above, we have that for this choice of $c$,
there exists a constant $\alpha > 1$ such that for any $j \leq \log^4 n$,
\[
\mathbb{P}\Big(\mathrm{Bin}(n-j,p) \leq k-1\Big) 
= (1+o(1))\mathbb{P}\Big(\mathrm{Poisson}(c) \leq k-1\Big) = (\alpha e c -o(1))^{-1}.
\]
The proof of \cref{i:kcp1} is similar to the argument appearing in \cite[Lemma 7 (b)]{A23} or \cite[Lemma 4.15]{AEKP25}. 
We show the statement for $G^+$, and the statement for $G^-$ follows similarly.
We first note that with probability $1- o\left(\frac{1}{n^2}\right)$, the maximum degree of $G^+$ is at most $\log^2 n$ in this regime of $p$. Consider then the $k$-core peeling process on $G^+$. Since the order of the largest component in the remainder can increase by a multiplicative factor of at most $\log^2 n$ each time a vertex is `peeled', if there is a component of $R^+ \coloneqq G^+ - V(K^+)$ of order at least $\log^4 n$ at the end of this process, there is some intermediary stage at which there is a component $C$ of peeled vertices of order between $\log^2 n$ and $\log^4 n$.
However, by construction there is a $k$-degenerate ordering of the vertices of $C$, and hence
each vertex in $C$ has at most $k-1$ neighbours in 
$V(G) \setminus C$. A simple union bound then shows that the probability that this occurs is at most
\begin{align*}
\sum_{j=\log^2 n}^{\log^4 n} \binom{n}{j} j^{j-2} p^{j-2} \left(\mathbb{P}\Big(\mathrm{Bin}(n-j,p) \leq k-1\Big)\right)^j &\leq \sum_{j=\log^2 n}^{\log^4 n} \frac{n^{2}}{j^2} (ec)^{j-2}  (\alpha ec -o(1))^{-j} \\
&= o(n^{-2}).
\end{align*}
To prove \cref{i:kcp2}, suppose for a contradiction that $\hc_k < c_k$ and let $c$ be such that $\hc_k < c < c_k$. Then, since $c < c_k$ by results in~\cite{PSW96}, we have that $K = \varnothing$ with probability $1-o(1)$. Together with \cref{i:kcp1}, this implies that with probability $1- o(1)$, we have that $G$ consists only of components of order at most $\log^4 n$. However, since $c > \hc_k > k > 1$, this contradicts \cref{l:basicprop}, which states that $G$ has a component of linear order with probability $1 - o\left(\frac{1}{n^2}\right)$.

Using \cref{i:kcp1} we can calculate similarly for \cref{i:kcp3},
\begin{align*}
    \expec{|W|^2 \ind{|W| \geq \ell} \indev{\cE}}  &\leq \sum_{j=\ell}^{\log^4 n} j^2 \binom{n-2}{j-2} j^{j-2}p^{j-2} \left(\mathbb{P}\Big(\mathrm{Bin}(n-j,p) \leq k-1\Big) \right)^j \\
 &\leq \sum_{j=\ell}^{\log^4 n}  j^2 e^2 \left( ec \right)^{j-2} (\alpha ec -o(1))^{-j} \coloneqq f_3(n,\ell),
 \end{align*}
where it is then clear that $\displaystyle \lim_{\substack{ \ell,n \to \infty \\  n \gg \ell}} f_3(n,\ell) = 0$.
\end{proof}

As mentioned in the introduction, one should expect a qualitatively similar statement to hold for any $p$ above the $k$-core threshold. Indeed, the work of \cite{COCKS17} suggests that the components of the mantle should resemble \emph{subcritical} (multi-type) Galton-Watson trees, although we were not able to extract an explicit statement that would imply \cref{lem:kcoreprop} from their work.

\subsection{Variance of the order of the giant and $k$-core}
For our arguments, it will be important that the variance of the order of the giant component (or $k$-core, respectively) is linear in $n$, which can be shown by elementary considerations. We include a proof for completeness.
\begin{lemma}\label{lem:variance}
Let $k \geq 3$ and let $\hc_k$ be as defined in \eqref{e:hc_k}. Let $p = \frac{c}{n}$, let $L$ be the largest component of $G(n,p)$, and let $K$ be the $k$-core of $G(n,p)$. 
\begin{enumerate}[label = \upshape{(V\arabic*)}]
    \item If $c > 1$, then $\variance{|L|} = \Omega(n)$.
    \item If $c > \hc_k$,
    then $\variance{|V(K)|} = \Omega(n)$. 
\end{enumerate}
\end{lemma}
\begin{proof}
We just prove the second statement, the first can be proved with an analogous argument.

We utilise a sprinkling argument, taking $p_1$ and $p_2$ such that $(1-p_1)(1-p_2) = 1-\frac{c}{n}$, where $p_1 = \frac{c'}{n}$, for some $c'$ such that $c > c' > \hc_k$ thus $p_2=\Theta\left(\frac{1}{n}\right)$. Let $G_1 \sim G(n,p_1)$ and $G_2 \sim G(n,p_2)$ be independent random graphs and let $G = G_1 \cup G_2$ so that $G \sim G(n,p)$. By \cref{lem:kcoreprop} \cref{i:kcp2}, we have that $\hc_k$ is above the $k$-core threshold $c_k$. Thus standard results on the $k$-core threshold (see, e.g., \cite{PSW96}) imply that whp\footnote{For a sequence of events $\mathcal{E}_n$, we say that $\mathcal{E}_n$ occurs \emph{with high probability} (whp) if $\prob{\mathcal{E}_n} \rightarrow 1$ as $n \rightarrow \infty$.} $G_1$ contains a $k$-core of linear order. Let us write $K_1$ for the $k$-core of $G_1$, where whp $|V(K_1)| = \Omega(n)$.

Let $A_1$ be the set of isolated vertices in $G_1$ and let $A_2 \subseteq A_1$ be the set of vertices which are isolated in $G - E_{G_2}(A_1,V(K_1))$. Then standard calculations show that whp $|A_2| = \Omega(n)$.

We let $K_2$ be the $k$-core in $G - E_{G_2}(A_{2},V(K_1))$. Then $K_1 \subseteq K_2 \subseteq K$ and we note that $|V(K)| = |V(K_2)| + Z$, where $Z$ is the number of vertices in $A_2$ adjacent to at least $k$ vertices of $K_2$ in $G_2$. Furthermore,  conditioned on $G_1$ and $G -E_{G_2}(A_{2}, V(K_1))$, noting that this determines $A_1,A_2,K_1$ and $K_2$, we have that $Z \sim \mathrm{Bin}(|A_2|,q)$ where $q = \prob{\mathrm{Bin}(|V(K_1)|,p_2) \geq k}$ thus, $\variance{Z} = q(1-q)|A_2|$.

Let $\mathcal{E}^*$ be the whp event that $|V(K_1)|,|A_2| = \Omega(n)$, which also implies that $q = \Omega(1)$, then a standard calculation (see, e.g., \cite[Lemma 6.4]{AEKP25}) shows that
\[
\variance{|V(K)|} \geq (1-o(1))\condvariance{Z}{\mathcal{E}^*} = \Omega(n),
\]
where the implicit constant depends on $c$ and $k$.
\end{proof}

\section{A central limit theorem for the order of the giant component}

Let $G \sim G(n,p)$ where $p = \frac{c}{n}$ and $c>1$. Let us define a function $\phi:V(G) \to \{0,1\}$ by
\[
\phi(v) \coloneqq \begin{cases} 1 &\text{ if } v \in V(L)\\ 0 &\text{ otherwise.} \end{cases}
\]
Note that $Z \coloneqq |L| = \sum_{v\in V(G)} \phi(v)$. For $\ell \in \mathbb{N}$, we also define a `localised' version of $\phi$ by
\[
\phi_{\ell}(v) \coloneqq \begin{cases} 1 &\text{ if } \partial_G(v,\ell) \neq \varnothing \text{ or } |B_G(v,\ell)|>\log^4 n 
\\ 0 &\text{ otherwise,} \end{cases}
\]
where $\partial_G(v,\ell)$ is the set of vertices at distance $\ell$ from $v$. Moreover, let 
$\tZ_\ell \coloneqq \sum_{v\in V(G)} \phi_\ell(v).$ 
Finally, viewing $\tZ_\ell$ as a weighted neighbourhood count with weights in $\{0,1\}$, let $\hZ_\ell$ be the truncation of $\tZ_\ell$ (as defined in \Cref{subsec:weighted neighbourhoods}).
\begin{lemma}\label{l:variancediff giant}
$\variance{Z - \tZ_\ell} \leq n \cdot f_4(n,\ell)$, where $\displaystyle \lim_{\substack{ \ell,n \to \infty \\  n \gg \ell}} f_4(n,\ell) = 0$.
\end{lemma}
\begin{proof}
By the definition of the double limit, we may assume that $\ell \leq \log n$.
Let us consider the effect of resampling an edge $e \coloneqq \{u,v\} \in \binom{V(G)}{2}$ and define as before $G^+ = G+e$ and $G^- = G-e$ and 
\begin{align} \label{eq:def_D}
    D \coloneqq \{w \in V(G) \colon \phi^+(w)-\phi_\ell^+(w) \neq \phi^-(w) - \phi_\ell^-(w)\}.
\end{align}
By Lemma \ref{lem:ESapp}, $\variance{Z - \tZ_\ell} \leq 2c n \expec{|D|^2}$. 
By \cref{l:basicprop}, we have
\begin{align} \label{e:var_bound}
    \variance{Z - \tZ_\ell} \leq 2c n \expec{|D|^2} \leq 2c n \expec{|D|^2 \indev{\cE}} + o(n),
\end{align}
where $\cE$ is the event as defined in \cref{l:basicprop}.

For $v \in V(G)$ and $*\in \{+,-\}$, let us define $C^*_v$ as in \Cref{l:basicprop} to be the component in $G^* - L^*$ containing $v$ (and $C^*_v=\varnothing$ if $v \in L^*$). Let $W \coloneqq C^-_u \cup C^-_v$. 
We note that in the event $\mathcal{E}$, each of $G^+,G^-$ has a unique linear order component $L^-$ and $L^+$ respectively. Moreover, since $G^-\subseteq G^+$, we have that $L^-\subseteq L^+$.

\begin{claim}\label{claim:D-W containment giant}
    In the event $\mathcal{E}$, we have that $D \subseteq W$.
\end{claim}
\begin{proofclaim}
Assume that the event $\mathcal{E}$ occurs. Note that if $w\notin W$ then $e$ does not meet the component of $w$ in $G^- - L^-$, that is, $C_w^- \cap \{u,v\} = \varnothing$. It follows that $G^-[C_w^-]=G^+[C_w^+]$.
This, and $L^-\subseteq L^+$ imply that $\phi^+(w)=\phi^-(w)$. 
Next we will show that $\phi^*_\ell(w)$ is completely determined by $G^*[C_w^*]$ for each $* \in \{-,+\}$. 
Indeed if $C_w^* \neq \varnothing$, then $\partial_{G^*}(w, \ell) = \partial_{G^*[C_w^*]}(w, \ell)$ and $B_{G^*}(w, \ell) = B_{G^*[C_w^*]}(w, \ell)$. On the other hand, if $C_w^*= \varnothing$, then $w \in L^*$ and thus we have $\partial_{G^*}(w, \ell) \neq \varnothing$ or $|B_{G^*}(w, \ell)| \geq |L^*| > \log^4 n$ implying that $\phi_\ell^*(w) = 1$.
Hence $G^-[C_w^-]=G^+[C_w^+]$ also implies $\phi^+_\ell(w)=\phi^-_\ell(w)$. Thus $\phi^+(w)-\phi^+_\ell(w)=\phi^-(w)-\phi^-_\ell(w)$, implying $w\notin D$. 
\end{proofclaim}

\begin{claim} \label{claim:W_small}
    In the event $\mathcal{E}$, we have that if $|W| < \ell$, then $D = \varnothing$.
\end{claim}
\begin{proofclaim}
Assume that the event $\mathcal{E}$ occurs and $|W| < \ell$. By \cref{claim:D-W containment giant}, it suffices to show that $w \notin D$ for all $w \in W$.
Let $w \in W$. Note that it suffices to show that for each $* \in \{-, +\}$, we have $\phi_\ell^*(w) = \phi^*(w)$ (since then $\phi^+(w)-\phi^+_\ell(w)=0=\phi^-(w)-\phi^-_\ell(w)$ which implies that $w \notin D$). Let $* \in \{-, +\}$. If $w \in L^*$, then $\phi^*(w) =1$ by definition and since $\cE$ occurs, we have $\partial_{G^*}(w, \ell) \neq   \varnothing$ or $|B_{G^*}(w, \ell)| > \log^4 n$, and therefore $\phi^*_\ell(w)=1$. Hence we may assume that $w \notin L^*$. Thus, by definition, we have $\phi^*(w) = 0$. Note that since $W \coloneqq C_u^- \cup C_v^-$ and $e = uv$, we have that $C_w^- \cup C_w^+ \subseteq W$. Hence $0 < |C_w^*| < \ell$. This implies that $\partial_{G^*}(w, \ell) = \varnothing$ and $|B_{G^*}(w, \ell)| \leq |C_w^*| < \ell \leq \log n$. Hence $\phi_\ell^*(w) = 0 = \phi^*(w)$.
\end{proofclaim}

By \cref{e:var_bound}, the two claims above, and Lemma \ref{l:basicprop} it follows that
\begin{align*}
    \variance{Z - \tZ_\ell} &\leq 2c n \expec{|D|^2 \indev{\cE}} + o(n) \leq 2 cn \expec{|W|^2\ind{|W| \geq \ell}\indev{\cE}} +o(n) \\
    &\leq 2n \cdot c f_2(n,\ell) +o(n) \eqqcolon n \cdot f_4(n,\ell),
\end{align*}
where $\displaystyle \lim_{\substack{ \ell,n \to \infty \\  n \gg \ell}}  f_4(n,\ell) = \lim_{\substack{ \ell,n \to \infty \\  n \gg \ell}} \left( 2c \cdot f_2(n,\ell) + o(1)\right)=0$.
\end{proof}

At this point we are essentially done. \Cref{l:variancediff giant} implies that $\tZ_\ell$ is a good approximation to $Z$, and Lemma \ref{l:trunc} implies that $\hZ_\ell$ is a good approximation to $\tZ_\ell$ and that $\hZ_\ell$ satisfies a central limit theorem. The rest of the proof consists of relatively straightforward calculations making the above precise, keeping careful track of the double limit over $\ell$ and $n$.

\begin{proof}[Proof of \Cref{t:giant}]
Firstly, from Minkowski's inequality (\cref{thm:prob_Minkowski}), we obtain
\begin{align}
\variance{Z - \hZ_\ell}  &= \variance{Z - \tZ_\ell +\tZ_\ell - \hZ_\ell} \nonumber \\
&\leq \left( \sqrt{\variance{Z - \tZ_\ell}}  + \sqrt{\variance{\tZ_\ell - \hZ_\ell}}\right)^2 \nonumber\\
&\leq n\cdot\left( \sqrt{ f_4(n,\ell)} + \sqrt{f_1(n,\ell)}\right)^2, \label{e:variancediff}
\end{align}
where at the last inequality we used \Cref{l:trunc,l:variancediff giant}.

Let us write $\sigma = \sqrt{\variance{Z}}$ and $\widehat{\sigma}_\ell = \sqrt{\variance{\hZ_\ell}}$, $Y \coloneqq Z - \expec{Z}$ and $\hY_\ell = \hZ_\ell - \expec{\hZ_\ell}$, and 
\[
X \coloneqq \frac{Z - \expec{Z}}{\sigma} =  \frac{Y}{\sigma}
, \qquad \hX_\ell \coloneqq \frac{\hZ_\ell - \expec{\hZ_\ell}}{\widehat{\sigma}_\ell} =  \frac{\hY_\ell}{\widehat{\sigma}_\ell}.
\]
Then, by \cref{lem:variance}, we have $\sigma = \Omega(\sqrt{n})$, and so by~\eqref{e:variancediff} together with \cref{thm:prob_Minkowski}, it follows that $\widehat{\sigma}_\ell = \Omega(\sqrt{n})$ as well. Moreover, by \eqref{e:prob_CS},
\begin{align}\label{e:variancequotient1}
\lim_{\substack{ \ell,n \to \infty \\  n \gg \ell}} \left| \frac{(\widehat{\sigma}_\ell){^2}}{\sigma{^2}} - 1 \right| 
&= \lim_{\substack{ \ell,n \to \infty \\  n \gg \ell}}  \frac{\left|\variance{Z}-\variance{\hZ_\ell}\right|}{\sigma{^2}} \nonumber\\ 
&\leq \lim_{\substack{ \ell,n \to \infty \\  n \gg \ell}} \left(\frac{\variance{Z - \hZ_\ell}}{\sigma{^2}} + 2\left(\frac{\variance{Z- \hZ_\ell}}{\sigma{^2}} \cdot \frac{\min\left\{\variance{Z}, \variance{\hZ_\ell}\right\}}{\sigma^2}\right)^{\frac{1}{2}} \right) \nonumber\\
&\leq \lim_{\substack{ \ell,n \to \infty \\  n \gg \ell}} \left(\frac{\variance{Z - \hZ_\ell}}{\sigma{^2}} + 2\left(\frac{\variance{Z- \hZ_\ell}}{\sigma{^2}}\right)^{\frac{1}{2}} \right)= 0.
\end{align}

Then, since $\sigma = \Omega(\sqrt{n})$, by \eqref{e:variancediff} and Chebyshev's inequality \begin{equation}\label{e:inprob0}
\left|\frac{Y - \hY_\ell}{\sigma}\right| \inprob 0 \qquad \text{ as $\ell,n \to \infty$.}
\end{equation}

Furthermore, for any $\ell \in \mathbb{N}$, by \Cref{l:trunc}, we have $\hX_\ell \indist \mathcal{N}(0,1)$ as $n \to \infty$ and so
\begin{equation}\label{e:indistnormal}
\hX_\ell \indist \mathcal{N}(0,1) \qquad \text{ as $\ell,n \to \infty$.}
\end{equation}

Hence, by Slutsky's Theorem (\Cref{thm:Slutsky}), using \eqref{e:variancequotient1} - \eqref{e:indistnormal}

\begin{align} \label{eq:limits}
X  = \frac{Y - \hY_\ell}{\sigma} + \frac{\widehat{\sigma}_\ell}{\sigma} \cdot \hX_\ell \indist 0 + 1 \cdot \mathcal{N}(0,1) = \mathcal{N}(0,1)  \qquad \text{ as $\ell,n \to \infty$.}
\end{align}

However, since $X$ does not depend on $\ell$, it follows that $X \indist \mathcal{N}(0,1)$ as $n \to \infty$, as claimed.
\end{proof}

\section{A central limit theorem for the order of the $k$-core} 

To emphasise the similarities between the proofs of \cref{t:giant,t:core} we will use similar variable names to the previous section (e.g., $R$, $C_w$, $\phi$, and $\phi_\ell$). Nevertheless, the reader should be careful to note that these objects are defined differently in these two sections.

Let $k \geq 3$ be given, let $c > \hc_k$, let $p  = \frac{c}{n}$ and let $G \sim G(n,p)$. Let $K$ be the $k$-core of $G$ and let $R \coloneqq G - V(K)$ be the mantle.
For each $v \in V(R)$, we let $C_v$ be the component of $R$ containing $v$ and for each $v \in V(K)$, we set $C_v = \varnothing$.

We define a function $\phi:V(G) \to \{0,1\}$ by
\[
\phi(v) \coloneqq \begin{cases} 1 &\text{ if } v \in V(K)\\ 0 &\text{ otherwise.} \end{cases}
\]
Note that $Y\coloneqq |V(K)| = \sum_{v\in V(G)} \phi(v)$. We also define a `localised' version as follows. Given $v \in V(G)$ and $\ell \in \mathbb{N}$, we define the \emph{$\ell$-local $k$-core} $K(v,\ell)$ as the maximal subset of $B_G(v,\ell)$ such that $\partial_G(v,\ell) \subseteq K(v,\ell)$ and for every $w \in K(v,\ell) \setminus \partial_G(v,\ell)$, we have that $w$ has at least $k$ neighbours in $K(v,\ell)$.  We note that we could equivalently define $K(v,\ell)$ in terms of a peeling process where, starting with $B_G(v,\ell)$ we recursively delete vertices in $B_G(v,\ell-1)$ of degree less than $k$.\footnote{Here the ball $B_G(v,\ell-1)$ is not updated dynamically in this process} In particular, if a vertex $w$ lies in a component $C_w^\ell$ of $G[B_G(v,\ell) \setminus K(v,\ell)]$, then there is some \emph{$\ell$-local $k$-degenerate ordering} of $C_w^\ell$, that is an ordering of the vertices of $C_w^\ell$ such that each vertex is adjacent to fewer than $k$ vertices which appear before it or are in $\partial_G(v,\ell)$. 
We then define
\[
\phi_{\ell}(v) \coloneqq \begin{cases} 1 &\text{ if } v \in K(v,\ell) \\ 0 &\text{ otherwise.} \end{cases}
\]
If $v$ lies in a component $C_v^\ell$ of $G[B_G(v,\ell) \setminus K(v,\ell)]$, then there is some $\ell$-local $k$-degenerate ordering $\Sigma$ of $C_v^\ell$. Note that $\Sigma$ can be extended to a $k$-degenerate ordering of $C_v$ in $G$ and therefore it does not include any vertices outside $C_v$. Thus $C_v^\ell \subseteq C_v$ and every $\ell$-local $k$-degenerate ordering of $C_v^\ell$ has the property that every vertex is adjacent to fewer than $k$ vertices which appear before it or are in $\partial_G(v,\ell)$ or are in $N_G(C_v^\ell)$. Moreover, note that if $C_v = \varnothing$ (which happens if and only if $v \in V(K)$), then $\phi_\ell(v) =1$.
It follows that $G[C_v]$ together with the edge boundary of $C_v$ in $G$ completely determine $C_v^\ell$. 

Let $\tY_\ell \coloneqq \sum_{v\in V(G)} \phi_\ell(v)$ and viewing $\tY_\ell$ as a weighted neighbourhood count 
with weights in $\{0,1\}$, let $\hY_\ell$ be the truncation of $\tY_\ell$ (as defined in \Cref{subsec:weighted neighbourhoods}).

\begin{lemma}\label{l:variancediff core}
$\variance{Y - \tY_\ell} \leq n \cdot f_5(n,\ell)$, where $\displaystyle \lim_{\substack{ \ell,n \to \infty \\  n \gg \ell}} f_5(n,\ell) = 0$.
\end{lemma}
\begin{proof}
     The proof of \Cref{l:variancediff core} follows as the proof of \Cref{l:variancediff giant} verbatim, with \Cref{lem:kcoreprop} in place of \Cref{l:basicprop} and the following minor alterations. In the proof of \cref{claim:D-W containment giant}, instead of the assertion that $G^-[C_w^-]=G^+[C_w^+]$ implies that $\phi^+_\ell(w)=\phi^-_\ell(w)$, we now use that since $G^-[C_w^-]=G^+[C_w^+]$ and the edge boundaries of $C^-_w$ and $C_w^+$ are identical in $G^-$ and $G^+$ respectively, we have that $\phi^+_\ell(w)=\phi^-_\ell(w)$.
     In the proof of \cref{claim:W_small}, we can instead use the following. If $w \in V(K^*)$ (i.e., $\phi^*(w) = 1$), then also $\phi^*_\ell(w) =1$ as observed above. Otherwise, we obtain $0 < |C_w^*| < \ell$ and thus
     $\partial_{G^*}(w, \ell) = \varnothing$. This implies that the $k$-degenerate ordering of $C_w^*$ is also a $\ell$-local $k$-degenerate ordering and thus $\phi^*(w) = 0 = \phi_\ell^*(w)$.
\end{proof}

\begin{proof}[Proof of \Cref{t:core}]
The proof of \Cref{t:core} follows as the proof of \Cref{t:giant} verbatim, with \Cref{l:variancediff core} in place of \Cref{l:variancediff giant}.
\end{proof}

\section*{Acknowledgements}
This research was funded in whole or in part by the Austrian Science Fund (FWF) [10.55776/P36131 (J.Erde),  10.55776/I6502 (M. Kang), 10.55776/ESP3863424 (M. Anastos)] and by the Swiss National Science Foundation (SNSF) [P500-2\_235474] (V. Pfenninger). For open access purposes, the authors have applied a CC BY public copyright license to any author accepted manuscript version arising from this submission. We thank the reviewers for helpful comments on the paper, and for bringing the particular form of Theorem \ref{thm:prob_Minkowski} to our attention.

\bibliographystyle{abbrv}
\bibliography{bib}

\end{document}